\documentclass[a4paper,11pt]{article}
\usepackage{amsmath,amsthm,amssymb,enumitem,xcolor}
\usepackage{tikz}
\usetikzlibrary{calc,decorations,decorations.pathreplacing}

\usepackage[hang]{footmisc}
\usepackage[nosort,nocompress,noadjust]{cite}

\usepackage[bookmarks=false,hyperfootnotes=false,colorlinks,
    linkcolor={red!60!black},
    citecolor={blue!50!black},
    urlcolor={blue!80!black}]{hyperref}

\renewcommand{\eqref}[1]{\hyperref[#1]{(\ref{#1})}}

\newlist{enumlist}{enumerate}{1}
\setlist[enumlist]{labelindent=0cm,label=(\roman*),labelwidth=4.5ex,labelsep=0.5ex,leftmargin=5ex,align=left,topsep=0.5ex,itemsep=1ex,parsep=1ex}

\newlist{itemlist}{itemize}{1}
\setlist[itemlist]{labelindent=0cm,label=$\bullet$,labelwidth=3.5ex,labelsep=0.5ex,leftmargin=4ex,align=left,topsep=0.5ex,itemsep=1ex,parsep=1ex}

\numberwithin{equation}{section}

{\theoremstyle{definition}\newtheorem{definition}{Definition}[section]

}

\newtheorem{proposition}[definition]{Proposition}
\newtheorem{lemma}[definition]{Lemma}
\newtheorem{theorem}[definition]{Theorem}
\newtheorem{corollary}[definition]{Corollary}

\mathchardef\mhyphen="2D

\makeatletter
\providecommand*{\diff}%
        {\@ifnextchar^{\DIfF}{\DIfF^{}}}
\def\DIfF^#1{%
        \mathop{\mathrm{\mathstrut d}}%
                \nolimits^{#1}\gobblespace
}
\def\gobblespace{%
        \futurelet\diffarg\opspace}
\def\opspace{%
        \let\DiffSpace\!%
        \ifx\diffarg(%
                \let\DiffSpace\relax
        \else
                \ifx\diffarg\[%
                        \let\DiffSpace\relax
                \else
                        \ifx\diffarg\{%
                                \let\DiffSpace\relax
                        \fi\fi\fi\DiffSpace} %

\makeatother

\begin{document}

\begin{center}
{\boldmath\LARGE\bf   A note on characterizations of relative amenability on finite von Neumann algebras}
\bigskip

{\sc by Xiaoyan Zhou{\renewcommand\thefootnote{1}\footnote{\noindent School of Mathematical Sciences, Dalian University of Technology. Dalian (China).\\ E-mail: doctoryan@mail.dlut.edu.cn and junshengfang@hotmail.com.}}and Junsheng Fang$^1$}
\end{center}

\medskip

\begin{abstract}\noindent
In this paper, we give another two characterizations of relative amenability on finite von Neumann algebras, one of which can be thought of as an analogue of injective operator systems. As an application, we prove a stable property of relative amenable inclusions. We prove that under certain assumptions, the inclusion $N=\int_{X} \bigoplus
N_{p} d \mu\subset M=\int_{X} \bigoplus
M_{p} d \mu$ is amenable if and only if $N_p\subset M_p$ is amenable almost everywhere.
\end{abstract}

\section{Introduction}

 One of the most important results in the theory of von Neumann algebras is the classification of injective factors (cases $II_1$, $II_\infty$, $III_{\lambda,~\lambda\neq 1}$) which was due to Connes \cite{cones}. Because of this work, and also with contributions from Choi, Effros, Lance, and Wassermann, mathematicians can characterize amenable type $II_1$ factors in many ways. Extending amenability of von Neumann algebras, Popa introduced the notion for relative amenability of von Neumann subalgebras \cite{Po86}. 
 Just like amenable $II_1$ factors having many characterizations (for example, semi-discretness, injectivity, property P of Schwartz, approximately finite-dimensional, etc.), Popa proved that relative amenability also has several equivalent descriptions \cite{Po86}. In fact, those characterizations are the analogues of the equivalent descriptions of amenability of von Neumann algebras. But there are still several equivalent descriptions of amenability for $II_1$
 factors for which he didn't find good analogue notions equivalent to the relative amenability. For example the semi-discretness, the isomorphism of $C^*(R,R')$ and $R\otimes_{min} R$, the approximate innerness of the flip automorphism on $R\otimes R$ and the existence of normal finite range completely positive maps tending to the identity. In \cite{M90}, Mingo proves the case of semi-discretness.

 We summarize the known characterizations of relative amenability on finite von Neumann algebras due to Popa \cite{Po86,MP03} and Mingo \cite{M90}.
 \begin{theorem}
Assume $M$ is a finite von Neumann algebra with a faithful normal trace $\tau$ and $N\subset M$ is a von Neumann subalgebra. The followings are equivalent.
\begin{enumerate}
\item The inclusion $N\subset M$ is amenable.
\item There exists a conditional expectation from $\langle M,e_N\rangle$ onto $M$.
\item $\langle M,e_N\rangle$ has a state that contains $M$ in its centralizer.
\item $H_N^1(M,X)=0$ for any dual $M$ bimodule $X$.
\item $M$ has a normal virtual $N$-diagonal.
\item The identity map on $M$ can be approximately factored by $E_N$.
\item For any $\varepsilon>0$ and any finite $F\subset U(M)$, there is a projection $f\in \langle M,e_N\rangle$ with $Tr(f)<\infty$ such that $\|ufu^*-f\|_2<\varepsilon \|f\|_{2,Tr}$ for all $u\in F$.
\end{enumerate}
\end{theorem}

Let $B(H)$ (resp. $B(K)$) be the bounded operators on a separable Hilbert space $H$ (resp. $K$). A linear subspace $E\subseteq B(H)$ is said to be an operator system if $E$ contains the identity and $E$ is self-adjoint. Recall that an operator system $E\subseteq B(H)$ is injective \cite{CE77} if given operator systems $L\subseteq S\subseteq B(K)$, each completely positive map $\phi:L\rightarrow E$
has a completely positive extension $\psi:S\rightarrow E$. In case $E$ is a von Neumann algebra, this notion of injectivity for von Neumann algebras coincides with that of injective von Neumann algebras \cite{EL77}. In Section 3, we give relative amenability a characterization which can be seen as an analogue of injective operator systems. The proof ideas are benefited from \cite{cones,CE77}.

In direct integral theory, suppose $M$ is a von Neumann algebra acting on a separable Hilbert space $H$ with center $\mathcal Z_M$, then there is a direct integral decomposition of
$M$ relative to $\mathcal Z_M$, i.e., there exists a locally compact complete separable metric measure space $(X,\mu)$ such that $H=\int_X \bigoplus H_p d\mu$ and $M=\int_X \bigoplus M_p d\mu$, where each $H_p$ is a separable Hilbert space for $p\in X$ and $M_p$ is a factor in $B(H_p)$ almost everywhere. If $N\subset M$ is a von Neumann subalgebra, then $N$ is a decomposable von Neumann algebra. We may write $N=\int_X \bigoplus N_p d\mu$, then by the definition of decomposable von Neumann algebras, it is not difficult to see that $N_p \subset M_p$ almost everywhere.
Connes in \cite{cones} proved some stability properties of the class of injective von Neumann algebras. Among many interesting results, he proved that
\begin{proposition}\cite{cones}
Let $H$ be a separable Hilbert space, $X$, a standard Borel space with probability measure $\mu$ and $p\rightarrow M_p$ a Borel map from $X$ to von Neumann subalgebras of $B(H)$. Then $M=\int_X \bigoplus M_p d\mu$ is injective if and only if almost all $M_p$ are injective.
\end{proposition}
Inspired by this result, it is natural to ask whether the inclusion $N\subset M $ is amenable imply $N_p\subset M_p$ is amenable almost everywhere and vice versa.
In Section 4, we prove that under certain assumptions, the above question has an affirmative answer.

\section{Preliminaries}
In this section, we recall some basic concepts that will be used later.
\subsection{The basic construction}
Let $M$ be a finite von Neumann algebra with a faithful normal trace $\tau$ and $N\subset M$ is a von Neumann subalgebra. We endow $M$ with the sesquilinear form $\langle x,y\rangle=\tau(x^*y)$, for $x,y\in M$. Denote by $L^2(M)$ the completion of $M$ with respect to $\langle \cdot,\cdot\rangle$. The corresponding $\|\cdot\|_\tau$ on $M$ is defined by $\|x\|_\tau=\sqrt{\tau(x^*x)}$. For $x,y\in M$, we put $\pi(x)\hat{y}=\widehat{xy}$. We have $$\|\pi(x)\hat{y}\|^2_\tau=\tau(y^*x^*xy)\leq \|x^*x\|\tau(y^*y)=\|x\|^2\|x\|^2_\tau.$$
Note that the unit vector $\hat{1}$ is cyclic and separating for $\pi(M)$. For simplicity, we write $x\hat{y}=\pi(x)\hat{y}$.
Let $J:L^2(M)\rightarrow L^2(M)$ by $x\hat{1}\mapsto x^*\hat{1}$, which extends to an antilinear surjective isometry of $L^2(M)$. We say that $J$ is \emph{the canonical conjugation operator} on $L^2(M)$.

Let $e_N$ denote the projection from $L^2(M)$ onto $L^2(N)$. The \emph{trace preserving conditional expectation $E_N$} from $M$ onto $N$ is defined to be the restriction $e_N|M$.
The \emph{basic construction} from the inclusion $N\subset M$ is defined to be the von Neumann algebra $\langle M,e_N\rangle:=(M, e_N)''$.
\begin{lemma}
Let $M$ be a finite von Neumann algebra with a faithful normal trace $\tau$ and $N\subset M$ is a von Neumann subalgebra. Then $e_N$ and $E_N$ have the following properties:
\begin{enumlist}
\item $e_N|N=E_N$ is a norm reducing map from $M$ onto $N$ with $E_N(1)=1$;
\item $E_N(bxc)=bE_N(x)b$ for all $x\in M,~b,c\in N$;
\item $\tau(xE_N(y))=\tau(E_N(x)E_N(y))=\tau(E_N(x)y)$ for all $x,y\in M$;
\item $e_Nxe_N=E_N(x)e_N=e_NE_N(x)$ for all $x\in M$;
\item $\langle M,e_N\rangle=JN'J$, $\langle M,e_N\rangle'=JNJ$, and the $^*$-subalgebra $Me_NM$ is weakly dense in $\langle M,e_N\rangle$.
\end{enumlist}
\end{lemma}

\subsection{Relative amenability}
Let $M$ be a finite von Neumann algebra with a faithful normal trace $\tau$. Given a normal completely positive map $\phi:M\rightarrow M$, we can use the Stinespring dilation to construct a correspondence which is denoted by $H_{\phi}$. Define on the linear space $H_{0}=M\otimes M$ a sesquiliniar form
$\langle x_{1}\otimes y_{1},x_{2}\otimes y_{2} \rangle_{\phi}=\tau(\phi(x^{*}_{2}x_{1})y_{1}y_{2}^{*})$, $\forall x_{1}, y_{1},x_{2},y_{2}\in M$.
It is easy to check that the complete positivity of $\phi$ is equivalent to the positivity of $\langle \cdot,\cdot\rangle_{\phi}$. Let $H_{\phi}$
be the completion of $H_{0}/\sim$, where $\sim$ is the equivalence modulo the null space of $\langle \cdot,\cdot\rangle_{\phi}$. Then $H_{\phi}$ is a correspondence of $M$ and the bimodule structure is given by $x(x_{1}\otimes y_{1})y=xx_{1}\otimes y_{1}y$. We call $H_{\phi}$ the correspondence of $M$
associated to $\phi$, see \cite{Po86}.

If we regard correspondences as $*$-representations, we can define a topology on these correspondences which is just the usual topology on the set of equivalent classes of representations of $N\otimes M^{op}$. Under this topology, we say a correspondence \emph{$H_{1}$ is weakly contained in $H_{2}$} if $H_{1}$ is in the closure of $H_{2}$.

\begin{definition}Let $M$ be a finite von Neumann algebra with a trace $\tau$, and $N$ be a von Neumann subalgebra of $M$, the inclusion \emph{$N\subset M$ is amenable} if $H_{id}$ is weakly contained in $H_{E_{N}}$, where $id$ is the identity map from $M$ to $M$ and $E_{N}$ is the faithful normal conditional expectation from $M$ onto $N$ preserving trace $\tau$.
\end{definition}
For more results on relative amenability, we refer the reader to \cite{Po86,A90,MP03,ZF17}.

\subsection{Direct integration}
General knowledge about direct integrals of separable Hilbert spaces and von Neumann algebras acting on separable Hilbert spaces can be found in \cite{D81,KR86}. Here we list a few definitions that will be needed in this paper.
\begin{definition}(\cite[14.1.1]{KR86})
If $X$ is a $\sigma$-compact locally compact space, $\mu$ is the completion of a Borel measure on $X$, and $\{H_p\}$ is a family of separable Hilbert spaces indexed by the points $p$ of $X$, we say that\emph{ a separable Hilbert space $H$ is the direct integral of $\{H_p\}$ over $(X,\mu)$ (we write $H=\int_X \bigoplus H_p d\mu(p)$)} when, to each $\xi\in H$, there corresponds a function $p\rightarrow \xi(p)$ on $X$ such that $\xi_p\in H_p$ for each $p$ such that
\begin{enumlist}
\item $p\rightarrow \langle \xi(p),\eta(p)\rangle$ is $\mu$-integral, when $\xi,\eta\in H$, and $\langle \xi,\eta\rangle=\int_X \langle \xi(p),\eta(p)\rangle d\mu(p)$.
\item If $u_p\in H_p$ for all $p\in X$ and $p\rightarrow \langle \mu_p,\eta(p)\rangle$ is $\mu$-integral for each $\eta\in H$, then there is a $\mu\in H$ such that $\mu(p)=\mu _p$ for almost $p$. We say that $\int_{X} \bigoplus H_{p} d \mu(p)$ and $p\rightarrow \xi(p)$ are the (direct integral) decomposition of $H$ and $\xi$, respectively.
\end{enumlist}
\end{definition}

\begin{definition}
Assume $H$ is the direct integral of $\{H_p\}$ over $(X,\mu)$.
\begin{enumlist}
\item (\cite[14.1.6]{KR86}) An operator $x\in B(H)$ is said to be \emph{decomposable} when there is a function $p\rightarrow x(p)$ on $X$ such that $x(p)\in B(H_p)$, and for each $\xi\in H$, $x_p\xi_p=(x\xi)(p)$ for almost $p$. If $x_p=f(p)I_p$, where $I_p$ is the identity operator on $H_p$, we say $x$ is \emph{diagonal}.
\item (\cite[14.1.12]{KR86}) A representation $\phi$ of a C*-algebra $A$ on $H$ is said to \emph{be decomposable over $(X,\mu)$} when there is representation $\phi_p$ of $A$ on $H_p$ such that $\phi(x)$
is decomposable for each $x\in A$ and $\phi(x)(p)=\phi_p(x)$ almost everywhere.
\item A state $\rho$ of $A$ is said to\emph{ be decomposable with decomposition $p\rightarrow \rho_p$} when $\rho_p$ is a positive linear functional on $A$ for each $p$, such that $\rho_p(x)=0$ when $\phi_p(x)=0$, $p\rightarrow \rho_p (x)$ is integral for each $x \in A$ and $\rho(x)=\int_X \rho_p(x)d\mu(p).$
\item (\cite[14.1.14]{KR86}) A von Neumann algebra $R$ on $H$ is said to \emph{be decomposable with decomposition $p\rightarrow R_p$ }when $R$ contains a norm-separable strong-operator-dense C*-algebra $A$ for which the identity representation $\iota$ is decomposable and such that $\iota_p(A)$ is strong-operator dense in $R_p$ almost everywhere.
\end{enumlist}
\end{definition}

\section{Two descriptions of relative amenability}
The purpose of this section is to give relative amenability two  characterizations.

Let $B(H)$ be the bounded operators on a Hilbert space $H$. A linear subspace $E\subseteq B(H)$ is said to be an operator system if $E$ contains the identity and $E$ is self-adjoint. Denote by $E_h$ the self-adjoint elements of $E$. The following result is inspired  by \cite{cones,CE77}.

\begin{theorem}\label{A}
Let $M$ be a finite von Neumann algebra and $N\subset M$ a von Neumann subalgebra. Then the following conditions are equivalent.
\begin{enumerate}
\item The inclusion $N\subset M$ is amenable.
\item For each $n\in \mathbb{N}$, $S=S^*\in \mathbb{M}_n(\mathbb{C})\otimes M$, each $\lambda=\lambda^*\in \mathbb{M}_n(\mathbb{C})$ such that $\lambda\otimes b\leq S$ for some $b=b^*\in \langle M,e_N\rangle$, there exists an $x=x^*\in M$ such that $\lambda\otimes x\leq S$.
\item Given each operator systems $E\subseteq F$, if completely positive map $\phi:E\rightarrow M$ has a completely positive extension $\psi:F\rightarrow \langle M,e_N\rangle$, then $\phi$ has a completely positive extension $\tilde{\phi}:F\rightarrow M$.
\end{enumerate}
\end{theorem}
\begin{proof}
$1\Rightarrow2$. By \cite[Theorem 3.23]{Po86}, the inclusion $N\subset M$ is amenable if and only if there exists a conditional expectation $E$ from $\langle M,e_N\rangle$ onto $M$. Let $id$ be the identity map on $\mathbb{M}_n(\mathbb{C})$, then $id\otimes E$ is a unital completely positive map since $E$ is a unital completely positive map. Then $id\otimes E$ is a conditional expectation from $\mathbb{M}_n(\mathbb{C})\otimes \langle M,e_N\rangle$ onto $\mathbb{M}_n(\mathbb{C})\otimes M$. Thus $0\leq (id\otimes E)(S-\lambda\otimes b)=S-\lambda\otimes E(b)$. Our result follows from the fact that $E(b)=E(b)^*\in M$.

The proof of $2\Rightarrow3$ is heavily inspired by the proof of Choi and Effros \cite[Theorem 3.4]{CE77}.  

$2\Rightarrow3$. Claim that for any fixed $b\in \langle M,e_N\rangle_h$, there exists a completely positive map $\omega:M+\mathbb{C}b\rightarrow M$ such that $\omega(x)=x$ for all $x\in M$.

\emph{Proof of the claim}.
We select a $r\in M_h$ and define $\omega(x+\lambda b)=x+\lambda  r$ for $x\in M$, $\lambda\in\mathbb{C}$. The case $b\in M_h$ is trivial. We may assume $b\notin M$.
For an operator system $E$, $s=[s_{ij}]\in \mathbb{M}_n(M)_h$, and $\lambda=[\lambda_{ij}]\in \mathbb{M}_n(\mathbb{C})_h$, define $W_E(s,\lambda)\subseteq E$ by $$W_E(s,\lambda)=\{r\in E_h:[s_{ij}+\lambda_{ij} r]\geq0\}.$$
Note that $\omega$ will be a completely positive map if and only if for each $m\in\mathbb{N_+}$, $[s_{ij} +\lambda_{ij}b]\in \mathbb{M}_m(M+\mathbb{C}b)^+$ implies that $[s_{ij} +\lambda_{ij}r]\in \mathbb{M}_m(M)^+,$ i.e.,
\begin{equation}\label{1}r\in \bigcap_{m\in\mathbb{N_+},s,\lambda} \left\{W_{M}(s,\lambda):b\in W_{\langle M,e_N\rangle}(s,\lambda),~s\in \mathbb{M}_m(M)_h,~\lambda\in \mathbb{M}_m(\mathbb{C})_h\right\},\end{equation}
since $b\in \langle M,e_N\rangle_h\setminus M$.

Note that $$W_{M}(s_1,\lambda_1)\cap W_{M}(s_2,\lambda_2)=W_{M}(s_1\oplus s_2,\lambda_1\oplus \lambda_2),$$ which implies the sets $\{W_{M}(s,\lambda)\}$ are directed by the ordering $\supseteq$. Moreover, by condition 2, if $$b\in W_{\langle M,e_N\rangle}(s_1\oplus \ldots \oplus s_k,\lambda_1\oplus \ldots\oplus\lambda_k),$$ then $$W_{M}(s_1\oplus \ldots \oplus s_k,\lambda_1\oplus \ldots\oplus\lambda_k)=W_{M}(s_1,\lambda_1)\cap \ldots \cap W_{M}(s_k,\lambda_k)\neq \varnothing,$$ which means that finite intersection of $\left\{W_{M}(s,\lambda):b\in W_{\langle M,e_N\rangle}(s,\lambda)\right\}$ is not empty.

If we take $s=\|b\|(1\oplus 1)$, $\lambda=1\oplus -1$, then we have $b\in W_{\langle M,e_N\rangle}(s,\lambda),$ and
$$W_M(s,\lambda)=\{r\in M_h:\|r\|\leq\|b\|\},$$ which is $\sigma$-weakly compact. 
Thus we obtain that the family  $$\left\{W_{M}(s,\lambda):b\in W_{\langle M,e_N\rangle}(s,\lambda)\right\}$$ have the finite intersection property, hence $$\bigcap_{m\in\mathbb{N_+},s,\lambda}\left\{W_{M}(s,\lambda):b\in W_{\langle M,e_N\rangle}(s,\lambda)\right\}\neq\varnothing.$$

To see this, put $W_0=W_M(\|b\|(1\oplus 1),1\oplus -1)$, assume $\cap\left\{W_{M}(s,\lambda):b\in W_{\langle M,e_N\rangle}(s,\lambda)\right\}=\varnothing$, then $\cup\left\{W_{M}(s,\lambda):b\in W_{\langle M,e_N\rangle}(s,\lambda)\right\}^\complement=W_0.$ Since all the sets $\left\{W_{M}(s,\lambda):b\in W_{\langle M,e_N\rangle}(s,\lambda)\right\}$ are not empty and $\sigma$-weakly closed, and $W_0$ is a $\sigma$-weakly compact set, there exist finite intersection of
$\left\{W_{M}(s,\lambda):b\in W_{\langle M,e_N\rangle}(s,\lambda)\right\}$  which is empty, thus we get a contradiction.

Thus we may get a completely positive map $\omega$ by selecting $r$ in the intersection. Hence we finish the proof of the claim.

Let $h=h^{*}\in F\setminus E$. By the assumption in condition 3, let $\psi: E+\mathbb{C}h \rightarrow \langle M,e_N\rangle$ be the completely positive map such that $\psi|_E =\phi$. Put $b=\psi(h)$, by the above claim, there exists a completely positive map $\omega: M+\mathbb{C}b\rightarrow M$ such that $\omega(x)=x$ for all $x\in M$. Then $\omega\circ \psi:E+\mathbb{C}h \rightarrow M$ is a completely positive extension of $\phi$. Do this step by step, and then using Zorn's Lemma
we get a completely positive extension $\tilde{\phi}:F\rightarrow M$.

$3\Rightarrow1$. Let $E=M$, $F=\langle M,e_N\rangle$ and $\phi$ be the identity map on $M$. It is obvious that the identity map on $M$ has a completely positive extension on $\langle M,e_N\rangle$, then by condition 3, $\phi$ has a completely positive extension $\tilde{\phi}:\langle M,e_N\rangle\rightarrow M$. Note that $\tilde{\phi}$ is a projection from $\langle M,e_N\rangle$ onto $M$ and it is contractive, by \cite[Theorem 1.5.10]{O}, $\tilde{\phi}$ is a conditional expectation.
\end{proof}

In the proof of $1\Rightarrow2$ of Theorem \ref{A}, in fact, we can still obtain that $\|x\|\leq \|b\|$, or, if we assume $\|b\|\leq1$, then $\|x\|\leq1$. Thus we have the following result.
\begin{corollary}\label{c1}
Let $M$ be a finite von Neumann algebra and $N\subset M$ a von Neumann subalgebra. Then the inclusion $N\subset M$ is amenable if and only if for each $n\in \mathbb{N}$, $X=X^*\in \mathbb{M}_n(\mathbb{C})\otimes M$, each $\lambda=\lambda^*\in \mathbb{M}_n(\mathbb{C})$ such that $\lambda\otimes b\leq X$ for some $b=b^*\in \langle M,e_N\rangle$ with $\|b\|\leq1$, there exists an $x=x^*\in M$ with $\|x\|\leq1$ such that $\lambda\otimes x\leq X$.
\end{corollary}

Condition 2 in Theorem \ref{A} may be regarded as a ``relative interpolation property": if one can solve a system of inequalities with an operator in $\langle M,e_N\rangle_h$, then one can do the same in $M_h$.

Consider the case $N=\mathbb{C}$. Then by Theorem \ref{A}, the inclusion $\mathbb{C}\subset M$ is amenable if and only if for each operator systems $E\subseteq F$, if each completely positive map $\phi:E\rightarrow M\subset B(H)$ has a completely positive extension $\psi:F\rightarrow B(H)$, then $\phi$ has a completely positive extension $\tilde{\phi}:F\rightarrow M$. Note that the ``if" part always holds since $\langle M,e_\mathbb{C}\rangle=B(H)$ and $B(H)$ is an injective operator system \cite[p.17]{O}. Recall that an operator system $E\subseteq B(H)$ is injective \cite{CE77} if given operator systems $L\subseteq S\subseteq B(K)$, each completely positive map $\phi:L\rightarrow E$ has a completely positive extension $\psi:S\rightarrow E$. Then we have the inclusion $C\subset M$ is amenable if and only if $M$ is an injective operator system. Thus  condition 3 in Theorem \ref{A} can be seen as an analogue of injective operator systems.

\section{A stable property of relative amenable inclusions}
In this part, we show an application of Theorem \ref{A}. First, we need to prove the following lemmas on the direct integrals.

In this section, we assume that all the von Neumann algebras have separable predual.

\begin{lemma}\label{lemma1}
If $M=\int_X \bigoplus M_p d\mu$ is a decomposable von Neumann algebra acting on a decomposable separable Hilbert space $H=\int_X \bigoplus H_p d\mu$ over $(X,\mu)$, and $N\subset M$ is a von Neumann subalgebra, then $N$ is decomposable. If we write $N=\int_X \bigoplus N_p d\mu$, then $N_p\subset M_p$ almost everywhere.
\end{lemma}
\begin{proof}
This result is well known to experts. We omit its proof here.
\end{proof}

\begin{lemma}\label{2lemma}
Assume $(M,\tau)$ is a finite von Neumann algebra. Let $M = \int_{X} \bigoplus
M_{p} d \mu$ and $L^2(M)=\int_{X} \bigoplus H_{p} d \mu$ be the
direct integral decompositions of $M$ and $L^2(M)$ relative to
$\mathcal Z_{ M}$ over $(X,\mu)$. Then there exist a faithful normal trace $\tau_p$ on $M_p$ almost everywhere such that for each $x\in M$, $\tau(x)=\int_{X} \tau_{p}(x_p) d \mu$, and $H_p=L^2(M_p,\tau_p)$ almost everywhere. Moreover, assume $N\subset M$ is a von Neumann subalgebra and write $N = \int_{X} \bigoplus
N_{p} d \mu$, if $e_N$ is decomposable relative to $\mathcal Z_{ M}$, then $(e_N)_p=e_{N_p}$ almost everywhere.
\end{lemma}
\begin{proof}
Since $M\hat{1}$ is dense in $L^2(M)$, by \cite[Lemma 14.1.3]{KR86}, we have $M_p\hat{1}_p$ is dense in $H_{p}$ for almost everywhere $p$.
Note that $x_p\mapsto \langle x_p \hat{1}_p,\hat{1}_p\rangle$ defines a positive normal linear functional on $B(H_p)$. We denote the restriction of this functional to $M_p$ by $\tau_p$. By \cite[Theorem 14.2.1]{KR86}, $\mathcal Z_{ M}$ is the algebra of diagonalizable operators relative to this decomposition. As in the proof of \cite[Lemma 14.1.19]{KR86}, it follows that, $\tau(x)=\int_{X} \tau_{p}(x_p) d \mu$ for each $x\in M$ and $\tau_p$ is, accordingly, faithful and tracial almost everywhere.

For almost every $p$ and for each $x\in M$, $\|x_p\hat{1}_p\|^2=\langle x_p \hat{1}_p,x_p \hat{1}_p\rangle=\|x_p \|^2_{\tau_p}$, then the linear operator $U_p:x_p\hat{1}_p\mapsto x_p $ extends to a unitary transformation from $H_p$ onto $L^2(M_p)$. Note that for almost every $p$ and for each $x_p,y_p\in M_p$, $U_p x_p U^*_p(y_p)=U_p x_py_p\hat{1}_p=x_py_p$, which means that $U_p M_p U^*_p=M_p$. Thus $\int_{X} \bigoplus H_{p} d \mu$ is unitarily equivalent to $\int_{X} \bigoplus L^2(M_p,\tau_p)d\mu$, and we may say $H_p=L^2(M_p,\tau_p)$ almost everywhere (in the sense of unitary equivalence) and $\hat{1}_p=id_p$ almost everywhere where $id=\int_{X} id_{p} d \mu$ is the identity operator in $M$.

From Lemma \ref{lemma1}, $N_p\subset M_p$ almost everywhere, so that $\tau_p$ is a faithful normal trace on $N_p$ almost everywhere.  Hence, by the discussion above, we can obtain $(L^2(N))_p=L^2(N_p) $ almost everywhere. Since $e_N$ is decomposable, $(e_N)_p$ is a projection almost everywhere from \cite[Lemma 14.1.20]{KR86}. 
By \cite[Lemma 14.1.3]{KR86}, $\{(e_N)_p(x_p),x\in M\}$ spans $(L^2(N))_p=L^2(N_p)$ almost everywhere. Thus $(e_N)_p$ is a projection from $L^2(M_p)$ onto $L^2(N_p)$ and $(e_N)_p=e_{N_p}$ almost everywhere.
\end{proof}

\begin{lemma}\label{lemma3}
Assume $(M,\tau)$ is a finite von Neumann algebra such that $M = \int_{X} \bigoplus
M_{p} d \mu$ and $L^2(M)=\int_{X} \bigoplus L^2(M_{p}) d \mu$ are decomposable over $(X,\mu)$ relative to $\mathcal Z_{ M}$, 
and $N\subset M$ is a von Neumann subalgebra with $\mathcal Z_{M}\subset N$. Then $\langle M, e_N\rangle$ is a decomposable von Neumann algebra and $\langle M, e_N\rangle = \int_{X} \bigoplus
\langle M_p, e_{N_{p}}\rangle d \mu$.
\end{lemma}
\begin{proof}
From \cite[14.2.1]{KR86}, $\mathcal Z_{M}$ is the algebra of the diagonalizable operators relative to this decomposition. Note that for each $t\in\langle M, e_N\rangle=JN'J$, we have $JtJ\in N'\cap B(L^2(M))\subset Z_{M}^{'} $, it follows that $JtJ$ is decomposable from \cite[14.1.10]{KR86}.
Write $JtJ=\int_{X} s_p d\mu$ with $s_p\in (N^{'})_{p}$, then $t=J\int_{X} s_p d\mu J$.  By \cite[14.1.24]{KR86}, $(N^{'})_{p}=(N_{p})^{'}$ almost everywhere.

Claim that $t=J\int_{X} s_p d\mu J=\int_{X}J_p s_p J_p d\mu $.

\emph{Proof of the claim}. For $x,y\in M$, we may write $x=\int_{X} x_p d\mu$, $y=\int_{X} y_p d\mu$. By \cite[Theorem 14.1.8]{KR86}, $(x_p)^*=(x^*)_p, (y_p)^*=(y^*)_p$ almost everywhere. By Lemma \ref{2lemma} we have
\begin{align*}
\langle J(\int_{X} s_p d\mu )J x \hat{1},y\hat{1}\rangle&=\overline{\langle (\int_{X} s_p d\mu) J x \hat{1},J y\hat{1}\rangle}\\
&=\overline{\langle (\int_{X} s_p d\mu) x^* \hat{1},y^*\hat{1}\rangle}\\
 &=\overline{ \int_{X}\langle s_p (x^*)_p \hat{1}_p,(y^*)_p\hat{1}_p\rangle d\mu}\\
 &=\overline{ \int_{X}\langle s_p J_p x_p \hat{1}_p,J_p y_p\hat{1}_p\rangle d\mu}\\
  &=\int_{X}\langle J_p s_p J_p x_p \hat{1}_p, y_p\hat{1}_p\rangle d\mu.
\end{align*}
Thus $\langle M, e_N\rangle$ is decomposable relative to $\mathcal Z_{M}$,
 and $t_p=J_p s_p J_p\in \langle M_p, e_{N_{p}}\rangle$ almost everywhere. By Lemma \ref{2lemma}, we have $\langle M, e_N\rangle_p=\langle M_p, (e_{N})_p\rangle=\langle M_p, e_{N_{p}}\rangle$ almost everywhere. Hence, we obtain that $\langle M, e_N\rangle = \int_{X} \bigoplus
\langle M_p, e_{N_{p}}\rangle d \mu$.
\end{proof}

Now we are proceed to show an application of Theorem \ref{A}. The proof idea comes from \cite[Proposition 6.5]{cones}
\begin{proposition}\label{propositoin}
Let $(M,\tau)$ be a finite von Neumann algebra and $N\subset M $ be a von Neumann subalgebra with $\mathcal Z_{M}\subset N$. Assume $M = \int_{X} \bigoplus
M_{p} d \mu$ and $L^2(M)=\int_{X} \bigoplus L^2(M_p) d \mu$ are the
direct integral decompositions of $M$ and $L^2(M)$ relative to
$\mathcal Z_{ M}$, where $(X, \mu)$ is a standard Borel space with probability measure $\mu$. Then the inclusion $N\subset M$ is amenable if and only if $N_p \subset M_p$ is amenable almost everywhere.
\end{proposition}
\begin{proof}
From  Corollary \ref{c1}, we know that the inclusion $N\subset M$ is not amenable if and only if there exist $n\in \mathbb{N}$, $S=S^*\in \mathbb{M}_n(\mathbb{C})\otimes M$, $\lambda=\lambda^*\in \mathbb{M}_n(\mathbb{C})$, $b=b^*\in \langle M,e_N\rangle$ with $\|b\|\leq 1$, such that $\lambda\otimes b\leq S$, then for each $x=x^*\in M$ with $\|x\|\leq1$, $\lambda\otimes x\nleq S$. Then by the Hahn-Banach separation theorem, there exist $\varepsilon>0$ and $\phi\in (\mathbb{M}_n(\mathbb{C})\otimes\langle M,e_N\rangle)^+_*$ such that $\phi(\lambda\otimes x)>\phi(S)+\varepsilon$ for the above $\lambda,x,S$.

For each $n\in\mathbb{N}$, let $(\lambda_{n,j})_{j\in\mathbb{N}}$ be a norm dense sequence in the self-adjoint part of $\mathbb{M}_n(\mathbb{C})$.
Then the inclusion $N\subset M$ is not amenable if and only if there exist $$n,j,q,k\in \mathbb{N},~S=S^*\in \mathbb{M}_n(\mathbb{C})\otimes M,~b=b^*\in \langle M,e_N\rangle,~\text{and~}\phi\in (\mathbb{M}_n(\mathbb{C})\otimes\langle M,e_N\rangle)^+_*$$ with $\|S\|\leq k$, $\|b\|\leq 1$ and $\|\phi\|\leq k$, such that $$\lambda_{n,j}\otimes b\leq S,~\phi(\lambda_{n,j}\otimes x)>\phi(S)+\frac{1}{q},\text{~for each~} x=x^*\in M, ~\|x\|\leq1.$$

We first assume that $L^2(M_p)=K$ almost everywhere for some separable Hilbert space $K$. This assumption is necessary since we will use a measure-theoretic result in the proof.

$\Leftarrow$ Assume $N_p \subset M_p$ is amenable almost everywhere, we show that the inclusion $N\subset M$ is amenable.

If the inclusion $N\subset M$ is not amenable, then we can choose $n,j,q,k,S,b,\phi$ as above. By Lemma \ref{lemma3}, $b$ is decomposable with $b_p\in \langle M_p, e_{N_{p}}\rangle$.
Since $M$ is decomposable, by \cite[p.188, p.201]{D81}, we have $$\lambda_{n,j}\otimes b=\int_{X} (\lambda_{n,j}\otimes b_p)d\mu,~\mathbb{M}_n(\mathbb{C})\otimes\langle M,e_N\rangle=\int_{X} \bigoplus (\mathbb{M}_n(\mathbb{C})\otimes\langle M,e_N\rangle_p) d \mu.$$ Note that $b$, $\lambda_{n,j}$ and $S$ are all self-adjoint, by \cite[Proposition 14.1.9]{KR86}, we deduce that $\lambda_{n,j}\otimes b_p\leq S_p$ almost everywhere.

Since  $N_p \subset M_p$ is amenable almost everywhere, we can find some $z_p=z_p^{*}\in M_p$ with $\|z_p\|\leq1$ such that $\lambda_{n,j}\otimes z_p\leq S_p$, and using the measurable selection technique (see \cite[14.3]{KR86} or \cite[Appendix V]{D81}, note that the measurable selection technique requires all the Hilbert spaces should be the same), we can choose these $z_p$'s to be Borel.

From \cite[Lemma 14.1.19]{KR86}
, we may assume $\phi$ is decomposable, so that $\phi_p(\lambda_{n,j}\otimes z_p)\leq\phi_p(S_p)$ almost everywhere since $\phi_p$ is a positive normal linear functional. Note that $z_p$ is essentially bounded and Borel, let $z=\int_{S} z_p d\mu$, \color{black}then by \cite[Proposition 14.1.9, Proposition 14.1.18]{KR86}, $z=z^*\in M$ and $\|z\|\leq k$. Thus $$\phi(\lambda_{n,j}\otimes z )=\int_X\phi_p(\lambda_{n,j}\otimes z_p)d\mu\leq\int_X\phi_p(S_p)d\mu=\phi(S).$$ This contradicts to the assumption that $\phi(\lambda_{n,j}\otimes x)>\phi(S)+\frac{1}{q}$ for each $x=x^*\in M$ with $\|x\|\leq k$.

$\Rightarrow$. Assume that the inclusion $N\subset M$ is amenable, we show that $N_p \subset M_p$ is amenable almost everywhere.

If there exists a non-negligible set $X_0$ such that $N_p \subset M_p$ is not amenable for $p\in X_0$, then there are integers $n,j,q,k$ such that the corresponding set of $p$'s is non-negligible, and, say, is equal to $Y\subseteq X_0$.
Then for each $p\in Y$, there exist $n,j,q,k\in \mathbb{N}$,
$$S_p=S_p^*\in \mathbb{M}_n(\mathbb{C})\otimes M_p, ~b_p=b_p^*\in \langle M_p,e_{N_p}\rangle, \text{~and~} \phi_p\in (\mathbb{M}_n(\mathbb{C})\otimes\langle M_p,e_{N_p}\rangle)^+_*$$
with $\|S_p\|\leq k$, $\|b_p\|\leq 1$ and $\|\phi_p\|\leq k$, such that $$\lambda_{n,j}\otimes b_p\leq S_p, ~\phi_p(\lambda_{n,j}\otimes x_p)>\phi(S_p)+\frac{1}{q}, \text{~for~each}~x_p=x_p^*\in M_p,~\|x_p\|\leq1.$$ Further more, using measurable selection technique (see \cite[14.3]{KR86} or \cite[Appendix V]{D81}), we can choose the corresponding families to be Borel.
Since these $b_p$'s, $S_p$'s and $\phi_p$'s are essentially bounded, there exist $b,S,\phi$ such that $$b=\int_Y b_p d\mu,~S=\int_Y S_p d\mu, ~\phi=\int_Y \phi_p d\mu.$$ Then from \cite[Proposition 14.1.18]{KR86}, Lemma \ref{lemma3}, and \cite[p.223]{D81}, we have $$S=S^*\in \mathbb{M}_n(\mathbb{C})\otimes M,~b=b^*\in \langle M,e_{N}\rangle,~\phi\in (\mathbb{M}_n(\mathbb{C})\otimes\langle M,e_{N}\rangle)^+_*.$$
Then by \cite[Lemma 14.1.8]{KR86}, we have $$\lambda_{n,j}\otimes b\leq S,~\phi(\lambda_{n,j}\otimes x)>\phi(S)+\frac{1}{q},\text{~for~each~}x=x^*\in M, \|x\|\leq1,$$which contradicts to the assumption that the inclusion $N\subset M$ is amenable.

Thus we prove our result for the case that $L^2(M_p)=K$ almost everywhere for some separable Hilbert space $K$.

For the general case, from \cite[Lemma 14.1.23]{KR86}, there exist a family of unitary transformations $\{U_p\}$ from $L^2(M_p)$ onto a separable Hilbert space $K$, such that $p\mapsto \xi_p$ is measurable for each $\xi\in L^2(M)$, and $p\mapsto U_p x_p U_p^*$ is measurable for each $x\in M$. If we write $U=\int_X U_p d\mu$, then $L^2(M)=\int_{X} \bigoplus L^2(M_{p}) d \mu$ is unitarily equivalent to $\int_{X} \bigoplus K d \mu$, and $M = \int_{X} \bigoplus
M_{p} d \mu$ is unitarily equivalent to $\int_{X} \bigoplus
U_p M_{p}U_p^* d \mu$. Note that for unitary transformations $U$ and $U_p$, the inclusion $UNU^*\subset UMU^*$ is amenable if and only if $N\subset M$ is amenable; the inclusion $U_pN_pU_p^*\subset U_pM_pU_p^*$ is amenable if and only if $N_p\subset M_p$ is amenable. Hence, from the preceding argument we have shown, the inclusion $N\subset M$ is amenable if and only if $N_p \subset M_p$ is amenable almost everywhere.
\end{proof}

{\bf Acknowledgment.} The second author was supported by the Project sponsored by the NSFC grant 11431011 and startup funding from Hebei Normal University.


\begin{thebibliography}{ABC90}\setlength{\itemsep}{-1mm} \setlength{\parsep}{0mm} \small
\bibitem[A90]{A90} Anantharaman-Delaroche, Claire, {\it On relative amenability for von Neumann algebras.} Compositio Math., {\bf74} (1990), no. 3, 333-352.

 \bibitem[BO08]{O}Brown, Nathanial. B; Ozawa, Narutaka,
  \newblock {\em $C^*$-algebras and finite-dimensional approximations},
  \newblock {Graduate Studies in Mathematics},
   88, American Mathematical Society, Providence, RI,
   2008, xvi+509.
\bibitem[CE77]{CE77} Choi, Man Duen; Effros, Edward G., Injectivity and operator spaces, {\it J. Functional Analysis}, { \bf 24} (1977), no. 2, 156-209.

\bibitem[C76]{cones} Connes, Alain, \newblock {\em Classification of injective factors. Cases II$_1$, II$_\infty$, III$_\lambda$, $\lambda\neq 1$.} \newblock {Ann. of Math.}
 (2), {\bf104} (1976), no. 1, 73-115.



\bibitem[D81]{D81} Dixmier, Jacques, {\it von Neumann algebras}, North-Holland Mathematical Library, 27. North-Holland Publishing Co., Amsterdam-New York, 1981. xxxviii+437 pp.

\bibitem[EL77]{EL77} Effros, Edward G.; Lance, E. Christopher, {\it  Tensor products of operator algebras,} Adv. Math., {\bf25} (1977), no. 1, 1-34.

\bibitem[KR86]{KR86} Kadison, Richard V.; Ringrose, John R., \newblock {\em Fundamentals of the theory of operator algebras. Vol. II. Advanced theory.}
Pure and Applied Mathematics, 100. Academic Press, Inc., Orlando, FL, 1986. pp. i¨Cxiv and 399-1074.

\bibitem[M90]{M90} Mingo, James A.,
Weak containment of correspondences and approximate factorization of completely positive maps, \newblock {J. Funct. Anal.}, {\bf89} (1990), no. 1, 90-105.

\bibitem[MP03]{MP03} Monod, Nicolas; Popa, Sorin, \newblock {\em On co-amenability for groups and von Neumann algebras.} \newblock {C. R. Math. Acad. Sci. Soc. R. Can.}, {\bf25} (2003), no. 3, 82--87.

\bibitem[Po86]{Po86} Popa, Sorin, Correspondences, INCREST Preprint, 56/1986.

\bibitem[ZF17]{ZF17} Zhou, Xiaoyan; Fang, Junsheng, {\it A note on relative amenability on finite von Neumann algebras}. Preprint.


\end{thebibliography}
\end{document}